\documentclass[12pt]{amsart}
\usepackage{amsmath}
\usepackage{amssymb}
\usepackage{amsthm}
\usepackage{enumerate}
\usepackage[mathscr]{eucal}
\usepackage{color}
\theoremstyle{plain}
\newtheorem{theorem}{Theorem}[section]
\newtheorem{prop}[theorem]{Proposition}

\newtheorem{lemma}{Lemma}[section]

\theoremstyle{definition}

\newtheorem{remark}{Remark}[section]

\usepackage[pagewise]{lineno}
\begin{document}
\title[On extreme contractions and the norm attainment set]{On extreme contractions and the norm attainment set of a bounded linear operator}
\author[Debmalya Sain]{Debmalya Sain}

\newcommand{\acr}{\newline\indent}

\address{\llap{\,}Department of Mathematics\acr
                              Indian Institute of Science\acr
															Bengaluru\acr
															Karnataka 560012\acr
                              INDIA}
\email{saindebmalya@gmail.com}

\thanks{The research of the author is sponsored by Dr. D. S. Kothari Postdoctoral Fellowship. The idea of this work germinated from a discussion with Prof. Gadadhar Misra, to whom the author is extremely indebted. The author feels elated to lovingly acknowledge the immense contribution of his aunt Ms. Sipra Das, towards the welfare of the family.} 


\begin{abstract}
In this paper we completely characterize the norm attainment set of a bounded linear operator on a Hilbert space. This partially answers a question raised recently in [\textit{D. Sain, On the norm attainment set of a bounded linear operator, accepted in Journal of Mathematical Analysis and Applications}]. We further study the extreme contractions on various types of finite-dimensional Banach spaces, namely Euclidean spaces, and strictly convex spaces. In particular, we give an elementary alternative proof of the well-known characterization of extreme contractions on a Hilbert space, that works equally well for both the real and the complex case. As an application of our exploration, we prove that it is possible to characterize real Hilbert spaces among real Banach spaces, in terms of extreme contractions on two-dimensional subspaces of it.

\end{abstract}

\maketitle

\section{Introduction.} 

The purpose of the present paper is to explore the norm attainment set of a bounded linear operator on a Hilbert space. We also study the extreme contractions on a finite-dimensional Banach space, from the point of view of operator norm attainment. Let us first establish the relevant notations and terminologies.\\
Let $\mathbb{X}, \mathbb{Y}$ denote Banach spaces defined over $ \mathbb{K}, $ the field of scalars. In this paper, unless otherwise mentioned, $ \mathbb{K} $ can be either the field of real numbers $ \mathbb{R}, $ or the field of complex numbers $ \mathbb{C}. $ Let $ B_{\mathbb{X}} $ and $ S_{\mathbb{X}} $ denote the unit ball and the unit sphere of $ \mathbb{X} $ respectively, i.e.,  $ B_\mathbb{X}=\{x \in \mathbb{X} : \|x\| \leq 1\} $ and $ S_\mathbb{X}=\{x \in \mathbb{X} : \|x\|=1\}. $ We reserve the symbol $ \mathbb{H} $ for Hilbert spaces. It is rather obvious that $ \mathbb{H} $ is also a Banach space, with respect to the usual norm induced by the inner product $ <,> $ on $ \mathbb{H}. $  For any two elements $x$ and $y$ in $\mathbb{X}$, $x$ is said to be orthogonal to $y$ in the sense of Birkhoff-James \cite{B}, written as $ x \perp_{B} y, $ if $\|x + \lambda y \| \geq \|x\|$ for all scalars $\lambda.$  If the norm on $ \mathbb{X} $ is induced by an inner product then Birkhoff-James orthogonality coincides with the usual inner product orthogonality, i.e., $ x \perp_{B} y $ if and only if $ <x,y> = 0. $ \\
Let $ \mathbb{L}(\mathbb{X}, \mathbb{Y}) $ denote the Banach algebra of all bounded linear operators from $ \mathbb{X} $ to $ \mathbb{Y}. $ We write $ \mathbb{L}(\mathbb{X}, \mathbb{Y}) = \mathbb{L}(\mathbb{X}), $ if $ \mathbb{X} = \mathbb{Y}. $ For $ T \in \mathbb{L}(\mathbb{X}, \mathbb{Y}), $ let $ M_T $ denote the set of unit vectors at which $ T $ attains norm, i.e., 
\[M_T = \{ x \in S_\mathbb{X} : \| Tx \| = \| T \| \}.\]

In this paper, in the context of a Hilbert space $ \mathbb{H}, $ we completely characterize $ M_T, $ for any $ T \in \mathbb{L}(\mathbb{H}). $ This answers a question raised very recently in \cite{Sb}, for the special case of Hilbert spaces. \\
It is easy to observe that for a non-zero $ T \in \mathbb{L}(\mathbb{X}, \mathbb{Y}), $ $ T $ is a scalar multiple of an isometry if and only if $ M_T = S_{\mathbb{X}}. $  It follows from the works of Koldobsky \cite{Kb} and Blanco and Turn$\breve{s}$ek \cite{Bb} that a non-zero $ T \in \mathbb{L}(\mathbb{X}, \mathbb{Y}) $ is a scalar multiple of an isometry if and only if $ T $ preseves Birkhoff-James orthogonality. In particular, it follows that a non-zero scalar multiple of an isometry always takes an orthogonal basis to an orthogonal basis (in the sense of Birkhof-James). In this paper, we prove an analogous result for any bounded linear operator defined on a finite-dimensional Hilbert space. Indeed, we prove that if $ \mathbb{H} $ is a finite-dimensional Hilbert space, then given any $ T \in \mathbb{L}(\mathbb{H}), $ there exists an orthonormal basis $ \mathbb{S} $ of $ \mathbb{H} $ such that $ T $ preserves orthogonality on $ \mathbb{S}. $

We also explore extreme contractions on various types of finite-dimensional Banach spaces, namely Euclidean spaces and strictly convex spaces. $ T \in  \mathbb{L}(\mathbb{X}, \mathbb{Y}) $ is said to be a contraction if $ \| T \| \leq 1. $ If, in addition, $ T $ is also an extreme point of the unit ball of $  \mathbb{L}(\mathbb{X}, \mathbb{Y}) $ then $ T $ is said to be an extreme contraction between $ \mathbb{X} $ and $ \mathbb{Y}. $ Extreme contractions of $ \mathbb{L}(\mathbb{H}), $ where $ \mathbb{H} $ is a complex Hilbert space, have been completely characterized by Kadison in \cite{K} as isometries or coisometries. Interestingly enough, for real Hilbert spaces, the same characterization of extreme contractions was obtained much later on by Grza$\acute{s}$lewicz in \cite{G}. In this work, we give an elementary alternative proof of the same, for the case of finite-dimensional Hilbert spaces. We would like to emphasis that while our proof is essentially finite-dimensional, it works all the same for both real and complex Hilbert spaces. \\
A Banach space $ \mathbb{X} $ is said to be strictly convex if for any two elements $ x, y \in \mathbb{X},$ $ \| x+y \| = \| x \| + \| y \| $ implies that $ y=kx, $ for some $ k \geq 0. $ Equivalently, $ \mathbb{X} $ is strictly convex if and only if every point of $ S_{\mathbb{X}} $ is an extreme point of the unit ball $ B_{\mathbb{X}}. $ 
We study extreme contractions on finite-dimensional strictly convex Banach spaces. We prove that if $ \mathbb{X} $ is an $ n- $dimensional  Banach space and $ \mathbb{Y} $ is any strictly convex Banach space, then $ T \in \mathbb{L}(\mathbb{X}, \mathbb{Y}) $ is an extreme contraction if $ \| T \|=1 $ and $ M_T $ contains $ n $ linearly independent vectors. As an application of this result, we prove that it is possible to characterize real Hilbert spaces among real Banach spaces, in terms of extreme contractions on two-dimensional subspaces of it. 

\section{Main results.}
Let us begin with a complete characterization of the norm attainment set of a bounded linear operator between Hilbert spaces. This answers a question raised in \cite{Sb}, for the special case of Hilbert spaces.

\begin{theorem}\label{theorem:characterization of M_T}
Let $ \mathbb{H}_1, \mathbb{H}_2 $ be Hilbert spaces and $ T \in \mathbb{L}(\mathbb{H}_1, \mathbb{H}_2). $ Given any $ x \in S_{\mathbb{H}_1}, $ $ x \in M_T $ if and only if the following two conditions are satisfied:\\
(i) $ <x,y> = 0 $ implies that $ <Tx,Ty> = 0, $\\
(ii) $ sup \{ \| Ty \| : ~~ \| y \| = 1,~ <x,y> = 0\} \leq \| Tx \|. $
\end{theorem}
\begin{proof}
Let us first prove the ``if" part. Let $ x \in M_T. $ It follows from the definition of $ \| T \| $ that $ (ii) $ holds true. We also note that every Hilbert space is smooth. Since Birkhoff-James orthogonality coincides with usual inner product orthogonality in a Hilbert space, it follows from Theorem $ 2.2 $ of \cite{Sb} that $ (i) $ holds true. This completes the proof of the ``if" part. \\
Let us now prove the ``only if" part. Let $ x \in S_{\mathbb{H}_1}  $ be such that $ (i) $ and $ (ii) $ are satisfied. Let $ z \in S_{\mathbb{H}_1} $ be chosen arbitrarily. It is easy to see that $ z $ can be written as $ z=\alpha x + h, $ where $ <x,h> = 0 $ and $ \alpha $ is a  scalar. If $ h=0 $ then $ 1 = \| z \| = | \alpha |.  $ Therefore, $ \| Tz \| = | \alpha | \| Tx \| = \| Tx \|. $ Let $ h \neq 0. $ We have, $ 1 = \| z \|^2 =~ <\alpha x + h,\alpha x + h > ~= |\alpha|^2 + \| h \|^2, $ since $ <x,h> = 0. $ Now, by virtue of $ (i) $ and $ (ii), $ we have,\\
$ \|Tz\|^2 =~ <\alpha Tx + Th,\alpha Tx + Th > =~ |\alpha|^2 \|Tx\|^2 + \|Th\|^2 =~ |\alpha|^2 \|Tx\|^2 + \|h\|^2 \|T(\frac{h}{\|h\|})\|^2 \leq |\alpha|^2 \|Tx\|^2 + \|h\|^2 \|Tx\|^2 = \|Tx\|^2.  $ This proves that given any $ z \in S_{\mathbb{H}_1}, $ $ \| Tz \| \leq \| Tx \|. $ In other words, it must be true that $ x \in M_T. $ This completes the proof of the ``only if" part and establishes the theorem.
\end{proof}

\begin{remark}

In view of Theorem $ 2.1, $ we would like to remark that for a bounded linear operator $ T $ between general Banach spaces $ \mathbb{X}, \mathbb{Y} $, obtaining a complete characterization of $ M_T $ seems to be much more difficult. To the best of our knowledge, this problem remains open.

\end{remark}

Next, we prove that if $ \mathbb{H} $ is a finite-dimensional Hilbert space then given any $ T \in \mathbb{L}(\mathbb{H}), $ there exists an orthonormal basis $ \mathbb{S} $ of $ \mathbb{H} $ such that $ T $ preserves orthogonality on $ \mathbb{S}. $

\begin{theorem}\label{theorem:orthogonality preserve}
Let $ \mathbb{H} $ be a finite-dimensional Hilbert space and $ T \in \mathbb{L}(\mathbb{H}). $ Then there exists an orthonormal basis $ \mathbb{S} = \{x_1, x_2, \ldots, x_n\} $ of $ \mathbb{H} $  such that $ T $ preserves orthogonality on $ \mathbb{S}, $ i.e., $ <Tx_i, Tx_j> = 0, $ whenever $ i \neq j. $
\end{theorem}
\begin{proof}
Since $ \mathbb{H} $ is finite-dimensional, $ S_{\mathbb{H}} $ is compact. Therefore, there exists $ x \in S_{\mathbb{H}} $ such that $ \| Tx \| = \| T \|. $ Let us choose $ x_1 = x. $ Consider the subspace $ \mathbb{H}_1 $ of $ \mathbb{H} $ given by $ \mathbb{H}_1 = \{ y \in \mathbb{H} ~:~ <x,y>=0 \}. $ Clearly, $ \mathbb{H}_1 $ is also a finite-dimensional Hilbert space. Let $ T_1 $ denote the restriction of $ T $ to $ \mathbb{H}_1. $ Once again, by applying the standard compactness argument, it is easy to see that there exists $ y \in S_{\mathbb{H}_1} $ such that $ \| T_1y \| = \| Ty \| \geq \| Tz \| $ for all $ z \in S_{\mathbb{H}_1}. $ Let us choose $ x_2 = y. $ We continue the process until we reach a one-dimensional subspace of $ \mathbb{H}. $ This gives us a basis $ \mathbb{S} = \{x_1, x_2, \ldots, x_n\} $ of $ \mathbb{H}. $ It is obvious from our construction that $ \mathbb{S} $ is an orthonormal basis of $ \mathbb{H}. $ Moreover, since orthogonality relation is symmetric in a Hilbert space, repeated applications of Theorem $ 2.1 $ shows that $ T $ preserves orthogonality on $ \mathbb{S}, $ i.e., $ <Tx_i, Tx_j> = 0, $ whenever $ i \neq j. $ This completes the proof of the theorem.
\end{proof}

\begin{remark}

Although $ T $ preserves orthogonality on the orthonormal basis $ \mathbb{S} $ of $ \mathbb{H}, $ it is not necessarily true that $ T(\mathbb{S}) $ is an orthonormal basis of $ \mathbb{H}. $ Indeed, $ T(\mathbb{S}) $ may not be a basis of $ \mathbb{H} $ at all. However, if $ T $ is invertible then $ T(\mathbb{S}) $ is also an orthogonal basis of $ \mathbb{H}. $ If, in addition, $ T $ is an isometry then certainly $ T(\mathbb{S}) $ is an orthonormal basis of $ \mathbb{H}. $

\end{remark}

Our next goal is to study extreme contractions on Banach spaces and Hilbert spaces. Let us begin with an easy but useful proposition.\\

\begin{prop}

Let $ \mathbb{X} $ be an $ n- $dimensional Banach space and $ \mathbb{Y} $ be any Banach space. Let $ T \in \mathbb{L}(\mathbb{X}, \mathbb{Y}) $ be such that $ \| T \| = 1, $  $ T $ attains norm at $ n $ linearly independent unit vectors $ x_1, x_2, \ldots, x_n $ and each $ Tx_i $ is an extreme point of $ B_{\mathbb{Y}}. $ Then $ T $ is an extreme contraction in $ \mathbb{L}(\mathbb{X}, \mathbb{Y}). $

\end{prop}

\begin{proof}
If possible, suppose that $ T $ is not an extreme contraction. Then there exists $ T_1, T_2 \in \mathbb{L}(\mathbb{X}, \mathbb{Y}) $ such that $ T_1, T_2 \neq T, $ $ \| T_1 \| = \| T_2 \| = 1 $  and $ T = tT_1 + (1-t)T_2,  $ for some $ t \in (0, 1).  $ Therefore, for each $ i \in \{1, 2, \ldots, n\}, $ we have, $ Tx_i=tT_1x_i+(1-t)T_2x_i. $ We also note that $ T_1x_i, T_2x_i \in B_{\mathbb{Y}}, $ as $ \|T_1\|=\| T_2 \|=1. $ Since $ Tx_i $ is an extreme point of $ B_{\mathbb{Y}}, $ it follows that $ T_1x_i=T_2x_i=Tx_i $ for each $ i \in \{1, 2, \ldots, n\}. $ However, this implies that $ T_1, T_2 $ agree with $ T $ on a basis of $ \mathbb{X} $ and therefore, $ T_1=T_2=T. $ This contradicts our initial assumption that $ T_1, T_2 \neq T $ and completes the proof of the proposition.
\end{proof}
Since in a strictly convex space, every point of the unit sphere is an extreme point of the unit ball, the proof of the following proposition is now immediate:

\begin{prop}

Let $ \mathbb{X} $ be an $ n- $dimensional Banach space and $ \mathbb{Y} $ be any strictly convex Banach space. Let $ T \in \mathbb{L}(\mathbb{X}, \mathbb{Y}) $ be such that $ \| T \| = 1 $ and $ T $ attains norm at $ n $ linearly independent unit vectors $ x_1, x_2, \ldots, x_n. $ Then $ T $ is an extreme contraction in $ \mathbb{L}(\mathbb{X}, \mathbb{Y}). $

\end{prop}

If $ \mathbb{H} $ is a Hilbert space then the extreme contractions of $ \mathbb{L}(\mathbb{H}) $ are precisely isometries and coisometries. We invite the reader to look through \cite{K} for the complex case and  \cite{G} for the real case. Here we give an alternate elementary proof of the same result, when the Hilbert space is finite-dimensional. We would like to remark that our proof, besides being elementary, remains valid for both real and complex cases. In order to prove the desired result, we require the following fact from \cite{Sb}(see \cite{Sc} for the complex case):\\

\begin{theorem}\label{theorem:orthogonality preserved norm attainment}
Let $ \mathbb{X}, \mathbb{Y} $ be smooth Banach spaces and $ T \in \mathbb{L}(\mathbb{X}, \mathbb{Y}). $ If $ x \in M_T $ then $ T $ preserves orthogonality at $ x, $ i.e., $ x \perp_B y  $ if and only if $ Tx \perp_B Ty. $
\end{theorem}

Now, the promised characterization of extreme contractions on finite-dimensional Hilbert spaces:

\begin{theorem}\label{theorem:extreme contractions on Euclidean spaces}
Let $ \mathbb{H} $ be a finite-dimensional Hilbert space. $ T \in \mathbb{L}(\mathbb{H}) $ is an extreme contraction if and only if $ T $ is an isometry.
\end{theorem}

\begin{proof}
Let us first prove the easier ``if" part. Let $ dim~ \mathbb{H} = n. $ Let $  T \in \mathbb{L}(\mathbb{H}) $ be an isometry. It is easy to observe that $ \|T\|=1 $ and $ M_T = S_{\mathbb{H}}. $ Therefore, there exists $ n $ linearly independent unit vectors at which $ T $ attains norm. Since every Hilbert space is strictly convex, it now follows from Proposition $ 2.4 $ that $ T $ is an extreme contraction in $ \mathbb{L}(\mathbb{H}). $ This completes the proof of the ``if" part.\\
Let us now prove the comparatively trickier ``only if" part.\\
Let $ T \in \mathbb{L}(\mathbb{H}) $ be an extreme contraction. It is easy to observe that $ \| T \|=1. $  Since $ \mathbb{H} $ is finite-dimensional, by the standard compactness argument, there exists a unit vector $ x_1 \in M_T. $ Applying Theorem $ 2.2, $ let us construct an orthonormal basis $ \mathbb{S} = \{x_1, x_2, \ldots, x_n\} $ of $ \mathbb{H} $ such that $ T $ preserves orthogonality on $ \mathbb{S}, $ i.e., $ <Tx_i, Tx_j> = 0, $ whenever $ i \neq j. $ Note that it follows from the construction of $ \mathbb{S} $ in the proof of Theorem $ 2.2 $ that $ \|Tx_i\| \geq \|Tx_j\|, $ if $ i < j. $  If $ x_i \in M_T $ for each $ i \in \{1, 2, \ldots, n\}, $ then applying Theorem $ 2.2 $ of \cite{Sa}, it is easy to see that $ M_T = S_{\mathbb{X}}. $ Since $ \|T\|=1, $ it follows that $ T $ is an isometry and we having nothing more to prove. Let us assume that $ x_1, \ldots, x_k \in M_T $ and $ x_{k+1}, \ldots, x_n \notin M_T. $ Let us now consider the following two cases and reach a contradiction in each of the cases to complete the proof of the theorem.\\
\\
Case I: $ \| Tx_{k+1} \| > 0. $ Let us choose $ \epsilon > 0 $ such that $ (1+\epsilon)^{2} \| Tx_{k+1} \|^2 < 1.  $ We would like to remark that since $ x_{k+1} \notin M_T, $ such a choice of $ \epsilon $ is always possible. Define a linear operator $ T_1 \in \mathbb{L}(\mathbb{H}) $ in the following way:\\
$ T_1x_i = Tx_i $ for each $ i \in \{ 1, \ldots, k \}, $\\
$ T_1x_{k+i} = (1+\epsilon) Tx_{k+i} $ for each $ i \in \{ 1, \ldots, n-k \}. $ \\
Note that $ T_1 \neq T. $ We claim that $ \| T_1 \|=1. $ Let $ z=\sum_{i=1}^{n} \alpha_ix_i \in S_{\mathbb{H}}, $ for some scalars $ \alpha_i. $ We have,$ \sum_{i=1}^{n} |\alpha_i|^{2}=1. $\\
If $ \alpha_{k+1}=\ldots=\alpha_{n}=0 $ then $ \| T_1z \|=\| Tz \|. $ On the other hand, if $ \alpha_{1}=\ldots=\alpha_{k}=0 $ then $ T_1z=\sum_{i=1}^{n-k}\alpha_{k+i}(1+\epsilon)Tx_{k+i}. $ Therefore, $ \|T_1z\|^{2}= \sum_{i=1}^{n-k}(1+\epsilon)^{2}|\alpha_{k+i}|^{2}\|Tx_{k+i}\|^{2} \leq \sum_{i=1}^{n-k}(1+\epsilon)^{2}|\alpha_{k+i}|^{2}\|Tx_{k+1}\|^{2} = (1+\epsilon)^{2}\|Tx_{k+1}\|^{2} < \|T\|^{2}. $ Let us assume that at least one of $ \alpha_{1},\ldots,\alpha_{k} $ (say $ \alpha_1 $) is nonzero and at least one of $ \alpha_{k+1},\ldots,\alpha_{n} $ (say $ \alpha_{k+1} $) is nonzero.\\
Choosing $ w=-\overline{\alpha_{k+1}}~x_1 + \overline{\alpha_1}~x_{k+1}, $ it is easy to see that $ <z,w>=0. $ However, an easy computation reveals that\\
$ <T_1z,T_1w> = -\alpha_1\alpha_{k+1}(1-(1+\epsilon)^{2}\|Tx_{k+1}\|^{2}) \neq 0. $ This proves that $ T_1 $ does not preserve orthogonality at such a $ z $ and therefore $ T_1 $ can not attain norm at such a $ z. $ Since $ T_1 $ must attain norm at some point of $ S_{\mathbb{H}} $ and $ \|T_1x_1\|=\|Tx_1\|=\|T\|=1, $ we conclude that $ \|T_1\|=1. $\\
Let us now define another linear operator $ T_2 \in \mathbb{L}(\mathbb{H}) $ in the following way:\\
$ T_2x_i = Tx_i $ for each $ i \in \{ 1, \ldots, k \}, $\\
$ T_2x_{k+i} = (1-\epsilon) Tx_{k+i} $ for each $ i \in \{ 1, \ldots, n-k \}. $\\
Clearly, $ T_2 \neq T. $ Similar to the case of $ T_1, $ it is easy to prove that $ \|T_2\|=1. $ Therefore, we have proved the following facts:\\
$ (i)~ \|T_1\|=\|T_2\|=\|T\|=1, $ $ (ii)~ T=\frac{1}{2}T_1+\frac{1}{2}T_1, $ $ (iii)~ T_1,~T_2 \neq T. $  However, this contradicts our initial assumption that $ T \in \mathbb{L}(\mathbb{H})$ is an extreme contraction.\\
\\
Case II: $ \| Tx_{k+1} \| = 0. $ It follows that $ \|Tx_{k+i}\|=0 $ for each $ i \in \{1,\ldots,n-k\}. $ We observe that for each $ i \in \{ 1,\ldots,k \}, $ $ (Tx_i)^{\perp}=\{y \in \mathbb{H}~ :~ <Tx_i,y>=0\} $ is a subspace of codimension $ 1 $ in $ \mathbb{H}. $ Therefore, it is easy to deduce that $ {\bigcap}_{i=1}^k (Tx_i)^{\perp} \neq \emptyset.$ Choose a fixed vector $ w \in {\bigcap}_{i=1}^k (Tx_i)^{\perp} \bigcap S_{\mathbb{H}}. $ Define a linear operator $ T_1 \in \mathbb{L}(\mathbb{H}) $ in the following way:\\
$ T_1x_i = Tx_i $ for each $ i \in \{ 1, \ldots, n \} \setminus \{k+1\}, $\\
$ T_1x_{k+1} = \frac{1}{2}w. $\\
Clearly, $ T_1 \neq T.  $ We claim that $ \| T_1 \|=1. $ Let $ z=\sum_{i=1}^{n} \alpha_ix_i \in S_{\mathbb{H}}, $ for some scalars $ \alpha_i. $ We have,$ \sum_{i=1}^{n} |\alpha_i|^{2}=1. $ We also have, $ \| T_1z \|^{2} = \sum_{i=1}^{k} |\alpha_i|^{2}+\frac{1}{4}| \alpha_{k+1} |^2 \leq \sum_{i=1}^{n} |\alpha_i|^{2}=1. $ Since $ \| T_1x_1 \|=\| Tx_1 \|=1, $ we must have $ \| T_1 \|=1. $ Define another linear operator $ T_2 \in \mathbb{L}(\mathbb{H}) $ in the following way:\\
$ T_2x_i = Tx_i $ for each $ i \in \{ 1, \ldots, n \} \setminus \{k+1\}, $\\
$ T_2x_{k+1} = -\frac{1}{2}w. $\\
As in the case of $ T_1, $ it is easy to observe that $ T_2 \neq T $ and $ \| T_2 \|=1. $ Therefore, we have proved the following facts:\\
$ (i)~ \|T_1\|=\|T_2\|=\|T\|=1, $ $ (ii)~ T=\frac{1}{2}T_1+\frac{1}{2}T_1, $ $ (iii)~ T_1,~T_2 \neq T. $  However, this contradicts our initial assumption that $ T \in \mathbb{L}(\mathbb{H})$ is an extreme contraction.\\
This establishes the theorem.
\end{proof}
As an application of the results obtained by us, we now obtain a characterization of real Hilbert spaces among all real Banach spaces, in terms of extreme contractions. First, let us prove the following lemma in order to obtain the desired characterization:
\begin{lemma}
Let $ \mathbb{X} $ be a two-dimensional real Banach space which is not strictly convex. Then there exists a linear operator $ T \in \mathbb{L}(\mathbb{X}) $ such that $ T $ is an extreme contraction in $ \mathbb{L}(\mathbb{X}) $ but $ T $ is not an isometry.
\end{lemma}

\begin{proof}
Since $ \mathbb{X} $ is not strictly convex, the unit sphere $ S_{\mathbb{X}} $ contains a closed straight line segment $ \mathbb{I}. $ We note that $ \mathbb{I} $ can be written as $ \mathbb{I}=\{x+\lambda y : \lambda_{1} \leq \lambda \leq \lambda_{2}\}, $ where $ x $ is a fixed interior point of $ \mathbb{I}, $ $ y $ is a fixed point on $ S_{\mathbb{X}} $ such that the straight line joining $ \theta $ and $ y $ is parallel to $ \mathbb{I} $ and $ \lambda_1, \lambda_2 $ are two fixed real numbers, one positive and the other negative. From the description of $ y, $ it is quite clear that $ x \perp_{B} y. $ Furthermore, we also have, $ v_1=x+\lambda_1 y $ and $ v_2=x+\lambda_2 y $ are extreme points of $ B_{\mathbb{X}}. $ Indeed, $ v_1, v_2 $ are the two end points of $ \mathbb{I} $ and therefore, $ v_1, v_2 $ must be extreme points of $ B_{\mathbb{X}}. $ Let $ w $ be any fixed extreme point of $ B_{\mathbb{X}}. $ Let us define a linear operator $ T \in \mathbb{L}(\mathbb{X}) $ in the following way:
 \[Tx=w,~Ty=0.\]
We claim that $ \|T\|=1. $ Clearly, $ \|T\| \geq \| Tx \|=\|w\|=1. $ On the other hand, let $ z=\alpha x + \beta y \in S_{\mathbb{X}}, $ where $ \alpha, \beta $ are scalars. We have, $ 1=\| \alpha x + \beta y \| \geq |\alpha|, $ since $ x \perp_{B} y. $ Therefore, $ \|Tz\|=\|\alpha w\|=|\alpha| \leq 1. $ This proves that $ \|T\|=1. $ It is now easy to observe that $ v_1, v_2 \in M_{T}. $ We also note that $ v_1, v_2 $ must be linearly independent and $ Tv_1 = Tv_2 = w. $ As $ w $ is an extreme point of $ B_{\mathbb{X}}, $ applying Proposition $ 2.3, $ we obtain that $ T $ is an extreme contraction in $ \mathbb{L}(\mathbb{X}). $ However, $ T $ can not be an isometry, since $ \|y\|=1 <0=\|Ty\|. $ This completes the proof of the lemma.   
\end{proof}

Let us now proceed towards establishing the promised characterization of real Hilbert spaces.
\begin{theorem}\label{theorem:characterization of Hilbert spaces}
A real  Banach space $ \mathbb{X} $ is a Hilbert space if and only if for every two-dimensional subspace $ \mathbb{Y} $ of $ \mathbb{X}, $ isometries are the only extreme contractions in $ \mathbb{L}(\mathbb{Y}). $
\end{theorem}

\begin{proof}
Let us first prove the ``only if" part. If $ \mathbb{X} $ is a Hilbert space then every two-dimensional subspace $ \mathbb{Y} $ of $ \mathbb{X} $ is also a Hilbert space. Therefore, it follows from Theorem $ 2.6 $ that isometries are the only extreme contractions in $ \mathbb{L}(\mathbb{Y}). $\\
Let us now prove the ``if" part. If possible, suppose that $ \mathbb{X} $ is not a Hilbert space. Then there exists a two-dimensional subspace $ \mathbb{Y} $ of $ \mathbb{X} $ such that $ \mathbb{Y} $ is not a Hilbert space. Let us first assume that $ \mathbb{Y} $ is strictly convex. As $ \mathbb{Y} $ is not a Hilbert space, it follows from Theorem $ 2.2 $ of \cite{Sa} that there exists a linear operator $ T \in \mathbb{L}(\mathbb{Y}) $ and two unit vectors $ e_1, e_2 \in S_{\mathbb{Y}} $ such that $ T $ attains norm at $ e_1, e_2 \in S_{\mathbb{Y}} $ but $ T $ does not attain norm at every point of $ span~\{e_1, e_2\}\bigcap S_{\mathbb{Y}}. $ It is immediate that $ T $ must be non-zero. We also note that $ M_{T}=M_{\frac{T}{\|T\|}}. $ Therefore, without loss of generality, we may and do assume that $ \|T\|=1. $ Since $ T $ does not attain norm at every point of $ span~\{e_1, e_2\}\bigcap S_{\mathbb{Y}}, $ it is easy to deduce that $ e_1, e_2 $ must be linearly independent and $ T $ can not be an isometry.  Since $ \mathbb{Y} $ is strictly convex and $ e_1, e_2 $ are linearly independent, it follows from Proposition $ 2.4 $ that $ T $ is an extreme contraction in $ \mathbb{L}(\mathbb{Y}). $ However, this contradicts our hypothesis as $ T $ is not an isometry.\\
Next, let us assume that $ \mathbb{Y} $ is not strictly convex. Lemma $ 2.1 $ ensures that there exists a linear operator $ T \in \mathbb{L}(\mathbb{Y}) $ such that $ T $ is an extreme contraction in $ \mathbb{L}(\mathbb{Y}) $ but $ T $ is not an isometry. This, once again, is a contradiction to our initial hypothesis. This establishes the theorem completely.

\end{proof}

As a concluding remark, we would like to add that while the characterization of extreme contractions between Hilbert spaces is well-understood, the scenario is far from complete in the more general setting of Banach spaces. In this paper, we have tried to illustrate the pivotal role played by the norm attainment set of a bounded linear operator in studying extreme contractions between Banach spaces. It is expected that the study will be further continued, in order to have a better understanding of extreme contractions and the norm attainment set of a bounded linear operator, in the context of general Banach spaces.





\end{document}